\documentclass[10pt,a4paper]{article}
\usepackage[utf8]{inputenc}
\usepackage{amsmath}
\usepackage{amsfonts}
\usepackage{amssymb}
\usepackage{amsthm}
\usepackage{enumitem}
\usepackage{comment}
\usepackage{mathtools}
\usepackage[hidelinks]{hyperref}

\title{Transfer operators for ultradifferentiable expanding maps of the circle.}
\author{Malo Jézéquel\footnote{Laboratoire de Probabilités, Statistique et Modélisation (LPSM),
CNRS, Sorbonne Université, Université de Paris,
4, Place Jussieu, 75005 Paris, France. The author is supported by the European Research Council (ERC) under the
European Union's Horizon 2020 research and innovation programme (grant agreement No 787304).}}

\providecommand{\keywords}[1]{\textbf{\textit{Keywords---}} #1}

\newcommand{\p}[1]{\left(#1\right)}

\newcommand{\va}[1]{\left| #1 \right|}
\newcommand{\n}[1]{\left\| #1 \right\|}
\newcommand{\R}{\mathbb{R}}
\newcommand{\C}{\mathbb{C}}
\newcommand{\Z}{\mathbb{Z}}
\newcommand{\N}{\mathbb{N}}

\newtheorem{thm}{Theorem}
\newtheorem{lm}[thm]{Lemma}
\newtheorem{cor}[thm]{Corollary}
\newtheorem{prop}[thm]{Proposition}
\theoremstyle{definition}

\newtheorem{rmq}[thm]{Remark}

\allowdisplaybreaks[1]

\begin{document}

\maketitle

\begin{abstract}
Given a $\mathcal{C}^\infty$ expanding map $T$ of the circle, we construct a Hilbert space $\mathcal{H}$ of smooth functions on which the transfer operator $\mathcal{L}$ associated to $T$ acts as a compact operator. This result is made quantitative (in terms of singular values of the operator $\mathcal{L}$ acting on $\mathcal{H}$) using the language of Denjoy--Carleman classes. Moreover, the \emph{nuclear power decomposition} of Baladi and Tsujii can be performed on the space $\mathcal{H}$, providing a bound on the growth of the dynamical determinant associated to $\mathcal{L}$.
\end{abstract}


\medskip

In the two previous articles \cite{lagtf,jezequelGlobalTraceFormula2019}, we introduced tools to study transfer operators and dynamical determinants for hyperbolic dynamics satisfying certain conditions of ultradifferentiability (i.e. hypotheses of regularity that are intermediate between $\mathcal{C}^\infty$ and real-analytic). In the present paper, we generalize our approach, producing a framework that allows to deal with any \emph{Denjoy--Carleman class} (we present Denjoy--Carleman classes in \S \ref{secdc}, see \cite{nqa} and references therein for a more complete survey). To make the exposition clearer, we restrict to the simplest case: expanding maps of the circle. More interesting cases, namely hyperbolic diffeomorphisms and Anosov flows, are dealt with respectively in \cite{lagtf} and in \cite{jezequelGlobalTraceFormula2019}. Since every $\mathcal{C}^\infty$ function belongs to some Denjoy--Carleman class, we shall prove in particular the following theorem (as a consequence of Theorem \ref{thm1quant} and Lemma \ref{lmelem}).

\begin{thm}\label{thmintro}
Let $T$ be a $\mathcal{C}^\infty$ expanding map of the circle. Then there exists a Hilbert space $\mathcal{H}$, continuously contained in $\mathcal{C}^\infty\p{\mathbb{S}^1}$ and that contains trigonometric polynomials as a dense subspace, such that the transfer operator 
\begin{equation}\label{eqdefL}
\begin{split}
\mathcal{L} : \varphi \mapsto \p{ x \mapsto \sum_{y : Ty = x} \frac{1}{\va{T'(y)}} \varphi(y)}
\end{split}
\end{equation}
defines a compact operator from $\mathcal{H}$ to itself.
\end{thm}

As far as we know, prior to our work on ultradifferentiable dynamics, constructions of spaces on which transfer operators are compact operators were only known for real-analytic hyperbolic dynamics (these constructions were pioneered by Ruelle \cite{Ruelle1}, see \cite{slipantschukCompleteSpectralData2017,slipantschukAnalyticExpandingCircle2013,bandtlowExplicitEigenvalueEstimates2008,naudRuelleSpectrumGeneric2012,bandtlowSpectralStructureTransfer2017,bandtlowLowerBoundsRuelle2019} for modern results on transfer operators for real-analytic hyperbolic dynamics).

Recall that a differentiable map $T$ from $\mathbb{S}^1 = \R / \Z$ to itself is said to be expanding if there is $\lambda > 1$ such that for all $x \in \mathbb{S}^1$ we have $\va{T'(x)} \geq \lambda$ (where the circle is parallelized in the usual way). Our method allows to tackle more general transfer operators than \eqref{eqdefL}, but we will focus on $\mathcal{L}$ in order to keep the exposition as simple as possible (we explain in the appendix how to deal with more general weights).

It can be shown that the spectrum of $\mathcal{L}$ acting on $\mathcal{H}$ from Theorem \ref{thmintro} is intrinsically defined by $T$ (this spectrum is called \emph{Ruelle spectrum} of $\mathcal{L}$). For instance, the non-zero eigenvalues of $\mathcal{L}$ are the inverses of the zeroes of the entire continuation of 
\begin{equation}\label{eqdefdet}
\begin{split}
d(z) \coloneqq \exp\p{- \sum_{n=1}^{+ \infty} \frac{1}{n} \textup{tr}^\flat \p{\mathcal{L}^n} z^n},
\end{split}
\end{equation}
where we set for $n \in \N^*$
\begin{equation}\label{eqdeftrace}
\begin{split}
\textup{tr}^\flat\p{\mathcal{L}^n} \coloneqq \sum_{x: T^n x = x} \frac{1}{\va{1 - \p{T^n}'(x)}}.
\end{split}
\end{equation}
The entire continuation of $d$ (that we still denote by $d$) is the \emph{dynamical determinant} of $T$. For the general theory of Ruelle spectrum and dynamical determinant for expanding and hyperbolic maps, see \cite{Bal2} and references therein.

We shall give bounds on the singular values of $\mathcal{L}$ acting on $\mathcal{H}$ depending on the Denjoy--Carleman class to which $T$ belongs (see Theorem \ref{thm1quant}). For the smallest classes, the operator $\mathcal{L}$ turns out to be trace class and in this case the dynamical determinant can be written as
\begin{equation*}
\begin{split}
d(z) = \det\p{I - z \mathcal{L}}.
\end{split}
\end{equation*}
When $\mathcal{L}$ is not known to be trace class, we will see that we can implement the \emph{nuclear power decomposition} from \cite{Tsu} to study the dynamical determinant, as stated in the following theorem (see \cite{Gohb,Groth} for the general theory of trace class and nuclear operators, notice that, on a Hilbert space, a nuclear operator of order $0$ is just a compact operators whose singular values are $p$-summable for all $p > 0$). This decomposition allows to write the dynamical determinant as a particular case of \emph{Weinstein--Aronszajn determinant} (see \cite[IV.\S 6]{Kato} and references therein).

\begin{thm}\label{thmintro2}
Let $T$ and $\mathcal{H}$ be as in Theorem \ref{thmintro}. There are two compact operators $\mathcal{L}_c$ and $\mathcal{L}_b$ from $\mathcal{H}$ to itself such that $\mathcal{L}_c$ is nuclear of order $0$, the spectral radius of $\mathcal{L}_b$ is $0$, the operator $\mathcal{L}$ is the sum of $\mathcal{L}_b$ and $\mathcal{L}_c$ and, for all $z \in \C$, we have
\begin{equation*}
\begin{split}
d(z) = \det\p{I - z\p{I - z \mathcal{L}_b}^{-1} \mathcal{L}_c}.
\end{split}
\end{equation*}
\end{thm}

This theorem can be made quantitative if $T$ belongs to a specific Denjoy--Carleman class, in particular we establish a bound on the growth of the dynamical determinant $d$ (see Proposition \ref{propbound}).

Notice that Theorems \ref{thmintro} and \ref{thmintro2} are already known when $T$ is real-analytic. In this case, $\mathcal{H}$ can be chosen to be a Hardy space and $\mathcal{L}_b$ to be $0$. The real-analytic case has been dealt with first by Ruelle in his pioneering paper \cite{Ruelle1}, in which he introduced the notion of dynamical determinant, and has been extensively studied recently \cite{bandtlowExplicitEigenvalueEstimates2008,naudRuelleSpectrumGeneric2012,slipantschukAnalyticExpandingCircle2013,bandtlowSpectralStructureTransfer2017,bandtlowLowerBoundsRuelle2019}.

The method that we develop here could probably apply to more general settings such as expanding maps on more general manifolds and hyperbolic diffeomorphisms or flows. In particular, one could probably use ideas from the present paper to improve the results from \cite{lagtf} and get a conjecturally optimal bound on the growth of dynamical determinants of Gevrey hyperbolic diffeomorphisms. See \S \ref{subsecga} for more details.

This paper is structured as follows. First we introduce the very elementary facts that we need about Denjoy--Carleman classes in \S \ref{secdc}. Then, we construct the space $\mathcal{H}$ from Theorem \ref{thmintro} and prove a quantitative version of Theorem \ref{thmintro} (namely Theorem \ref{thm1quant}) in \S \ref{sech}. In \S \ref{secnpd}, we implement the nuclear power decomposition in our space $\mathcal{H}$ to prove Theorem \ref{thmintro2}. We discuss some examples in \S \ref{secex}. Finally, we explain how to deal with weighted transfer operators in the appendix.

\section{Denjoy--Carleman classes}\label{secdc}

The interested reader may consult \cite{nqa} and references therein for a more complete introduction to the topic of Denjoy--Carleman classes.

Let $M = \p{M_k}_{k \in \N}$ be an increasing and logarithmically convex sequence of positive real numbers such that $M_0 = 1$. Recall that the fact that $M$ is logarithmically convex means that
\begin{equation*}
\begin{split}
\forall k \in \N^* : M_k^2 \leq M_{k-1} M_{k+1}.
\end{split}
\end{equation*}

The sequence $M$ is now fixed until \S \ref{secex}. We say that a $\mathcal{C}^\infty$ function $f : \mathbb{S}^1 \to \C$ is in the Denjoy--Carleman class $\mathcal{C}^M$ if there are constants $ C,R > 0$ such that for all $k \in \N$ and $x \in \mathbb{S}^1$ we have
\begin{equation}\label{eqdefdc}
\begin{split}
\va{f^{(k)}(x)} \leq C R^k k! M_k.
\end{split}
\end{equation}
We will not define what it means for a map $T : \mathbb{S}^1 \to \mathbb{S}^1$ to belong to the class $\mathcal{C}^M$. We shall rather assume when needed that the derivative $T' : \mathbb{S}^1 \to \R$ belongs to the class $\mathcal{C}^M$. Since the class $\mathcal{C}^M$ does not need to be closed under differentiation, it does not necessarily imply that $T$ belongs to $\mathcal{C}^M$ (for any reasonable definition).
 
We will use Denjoy--Carleman classes in a very basic way and, consequently, we do not need any fact from the general theory of Denjoy--Carleman classes. However, the proof of Lemma \ref{lmfonda} below is very similar to the proof of the stability by composition of the class $\mathcal{C}^M$ (that relies on the fact that $M$ is logarithmically convex).

To the class $\mathcal{C}^M$, we associate the function $w = w^M$ on $\R_+^*$ defined by 
\begin{equation}\label{eqdefw}
\begin{split}
\forall x \in \R_+^* : w(x) \coloneqq \inf_{k \in \N} x^k k! M_k.
\end{split}
\end{equation}
The function $w$ will play a fundamental role in estimates on singular values and norms of operators appearing in the nuclear power decomposition of the transfer operator. We are not aware of any reference introducing precisely the function $w$, but it seems common to introduce similar objects adapted to a particular problem (see for instance \cite[(1.1)]{furdosAlmostAnalyticExtensions2019}). The following lemma lists basic properties of the function $w$.

\begin{lm}\label{lmpropw}
The function $w$ is continuous and increasing from $\R_+^*$ to itself. Moreover, $w$ vanishes at all orders in $0$, i.e. for all $\alpha \in \R$ we have $x^\alpha w(x) \underset{x \to 0}{\to} 0$. If $\mu \in \left]0,1\right[$ then $\frac{w(\mu x)}{w(x)} \underset{x \to 0}{\to} 0$. If in addition $\gamma > 1$ is such that there is $C > 0$ such that for all $k \in \N$ we have
\begin{equation}\label{eqgentil}
\begin{split}
(k+1) M_{k+1} \leq C \gamma^k M_k,
\end{split}
\end{equation}
then, if $\mu \in \left]0,1\right[$, there is a constant $C'$ such that for all $x > 0$ we have
\begin{equation}\label{eqpoly}
\begin{split}
\frac{w(\mu x)}{w(x)} \leq C' x^\delta,
\end{split}
\end{equation}
where $\delta = - \frac{\log \mu}{\log \gamma}$.
\end{lm}

\begin{proof}
Since $w$ is defined as an infimum of increasing functions, $w$ is increasing. Since $w(x)$ is smaller than $x^k k! M_k$ for all $k$, it is clear that $w$ vanishes at all orders in $0$.

If $x \in \R_+^*$, since $ x^k k! M_k \underset{k \to + \infty}{\to} + \infty$, the infimum in the definition of $w(x)$ is attained by a finite number of integers $k$. Denote by $k(x)$ the largest integer that realizes this infimum. Notice that if $\ell \leq m$ then the logarithmic derivative of $x \mapsto x^\ell \ell! M_\ell$ is smaller than that of $x \mapsto x^m m! M_m$. Consequently, the function $x \mapsto k(x)$ is decreasing. 
Thus if $x_0 > 0$ then for all $x > x_0$ since $k(x) \leq k(x_0)$ we have
\begin{equation*}
\begin{split}
w(x) = x^{k(x)} k(x)! M_{k(x)} = \min_{n=0,\dots,k(x_0)} x^n n! M_n,
\end{split}
\end{equation*}
and consequently $w$ is continuous on $\left]x_0,+ \infty\right[$. Since $x_0 > 0$ is arbitrary, $w$ is continuous on $\R_+^*$.

Let $\mu$ be an element of $\left]0,1\right[$. Notice that for all $x > 0$ we have
\begin{equation}\label{eqdecayw}
\begin{split}
\frac{w(\mu x)}{w(x)} = \frac{w(\mu x)}{x^{k(x)} k(x)! M_{k(x)}} \leq \frac{\p{\mu x}^{k(x)} k(x)! M_{k(x)}}{x^{k(x)} k(x)! M_{k(x)}} = \mu^{k(x)},
\end{split}
\end{equation}
and since it is clear that $k(x) \underset{x \to 0}{\to} + \infty$, we get that $\frac{w(\mu x)}{w(x)} \underset{x \to 0}{\to} 0$. Assume now that \eqref{eqgentil} holds. Notice that if $0 < x < \frac{1}{C}$ then
\begin{equation*}
\begin{split}
\p{\frac{x}{\gamma}}^{k(x) +1} \p{k(x) + 1}! M_{k(x) +1} \leq x^{k(x)} k(x)! M_{k(x)},   
\end{split}
\end{equation*}
and thus we have
\begin{equation*}
\begin{split}
k\p{\frac{x}{\gamma}} \geq k(x) + 1.
\end{split}
\end{equation*}
Now, if $0 < x < \frac{1}{C}$, letting $n$ be the largest integer such that $\gamma^{n} x < \frac{1}{C}$, we find that
\begin{equation*}
\begin{split}
k(x) = k\p{\frac{\gamma^n x}{\gamma^n}} \geq k\p{\gamma^n x} + n \geq n \geq - \frac{\log x}{\log \gamma} - a,
\end{split}
\end{equation*}
where $a = \frac{\log \p{\gamma C}}{\log \gamma}$. Thus by \eqref{eqdecayw} we find that if $0 < x < \frac{1}{C}$ then
\begin{equation*}
\begin{split}
\frac{w(\mu x)}{w(x)} \leq C' x^{\delta},
\end{split}
\end{equation*}
where $C ' = \mu^{- a}$.
\end{proof}

We end this section with a lemma that implies in particular that every $\mathcal{C}^\infty$ function on the circle belongs to some Denjoy--Carleman class. It allows us to deduce Theorems \ref{thmintro} and \ref{thmintro2} from their quantitive versions Theorem~\ref{thm1quant} and  Propositions \ref{propfastdecay}, \ref{propnucpart} and \ref{propfact}. We omit the elementary proof.

\begin{lm}\label{lmelem}
Let $\p{A_k}_{k \in \N}$ be a sequence of non-negative real numbers. Then there are a constant $C > 0$ and an increasing and logarithmically convex sequence $\p{B_k}_{k \in \N}$ of positive real numbers such that $B_0 = 1$ and
\begin{equation*}
\begin{split}
\forall k \in \N : A_k \leq C B_k.
\end{split}
\end{equation*}
\end{lm}
%

\section{Construction of the space $\mathcal{H}$}\label{sech}

Let $T$ be an expanding map of the circle, that is there is $\lambda>1$ such that for all $x \in \mathbb{S}^1$ we have $\va{T'(x)} > \lambda$. We assume in addition that $T'$ belongs to the class $\mathcal{C}^M$. We recall that the transfer operator $\mathcal{L}$ associated to $T$ is defined by \eqref{eqdefL}. We shall explain in the appendix how to deal with more general transfer operators. The function $w=w^M$ (defined by \eqref{eqdefw}) allows us to state a quantitative version of Theorem \ref{thmintro} (see \cite[\S IV.2]{Gohb} for definition of singular values).

\begin{thm}\label{thm1quant}
For every $\theta \in \left]1,\lambda\right[$ there are constants $C,A > 0$ and a Hilbert space $\mathcal{H}$ continuously contained in $\mathcal{C}^\infty\p{\mathbb{S}^1}$ and containing trigonometric polynomials as a dense subspace, such that $\mathcal{L}$ defines a compact operator from $\mathcal{H}$ to itself. Moreover, if $\p{\sigma_k}_{k \in \N}$ is the sequence of singular values of $\mathcal{L}$ acting on $\mathcal{H}$ then we have
\begin{equation}\label{eqboundsing}
\begin{split}
\forall k \in \N^* : \sigma_k \leq C \sup_{0 < x \leq \frac{1}{k}} \frac{w\p{ A x}}{w\p{\theta Ax}}.
\end{split}
\end{equation}
\end{thm}

Let us start the proof of Theorem \ref{thm1quant}. Let $\theta \in \left]1,\lambda\right[$ be fixed once for all. If $n \in \Z$, we write $e_n$ for the function on the circle $e_n : x \mapsto e^{2 i \pi n x}$. Define the family $\p{\pi_n}_{n \in \N}$ of orthogonal projectors on $L^2\p{\mathbb{S}^1}$ by 
\begin{equation*}
\begin{split}
\pi_n u = \left\{ \begin{array}{cc}
\langle u , e_0 \rangle_{L_2} e_0 & \textrm{ if } n = 0 \\
\sum_{\theta^{n-1} \leq \va{k} < \theta^{n}} \langle u, e_k \rangle_{L^2} e_k & \textrm{ otherwise } 
\end{array} \right. .
\end{split}
\end{equation*}
In order to give the definition of the space $\mathcal{H}$ from Theorem \ref{thm1quant}, we need to state a technical but fundamental result. 

\begin{lm}\label{lmmixedterm}
There are constants $C,R > 0$ such that for all $m,n \in \N$ and $u \in L^2\p{\mathbb{S}^1}$ such that $m \geq n$ we have
\begin{equation*}
\begin{split}
\n{\pi_m \mathcal{L} \pi_n u}_{L^2} \leq C w\p{\frac{R}{\theta^m}} \theta^{\frac{m + n}{2}} \n{\pi_n u}_{L^2}. 
\end{split}
\end{equation*}
\end{lm}

We can now define $\mathcal{H} = \mathcal{H}_{\theta,R,M}$ as the space of $u \in L^2\p{\mathbb{S}^1}$ such that ($R$ is the constant from Lemma \ref{lmmixedterm}):
\begin{equation}
\begin{split}
\sum_{m \in \N}\lambda^{-2m} w\p{\frac{R}{\theta^{m-1}}}^{-2} \n{\pi_m u}_{L^2}^2 < + \infty.
\end{split}
\end{equation}
It is easily seen that the square root of the quantity above defines a norm for which $\mathcal{H}$ is a Hilbert space. From Lemma \ref{lmpropw}, the quantity $\lambda^{-m} w\p{\frac{R}{\theta^{m-1}}}^{-1}$ tends to infinity faster than any geometric sequence when $m$ tends to infinity. Consequently, the space $\mathcal{H}$ is continuously contained in $\mathcal{C}^\infty\p{\mathbb{S}^1}$. One can check easily that trigonometric polynomials form a dense subspace of $\mathcal{H}$.

Before proving Lemma \ref{lmmixedterm}, we need another technical result.

\begin{lm}\label{lmfonda}
There are constants $C,R > 0$ such that for all $k,\ell \in \Z$ such that $\va{k} > \theta^{-1} \va{\ell}$ we have
\begin{equation*}
\begin{split}
\va{\langle \mathcal{L} e_\ell, e_k\rangle_{L^2}} \leq C w\p{\frac{R}{\va{k}}}.
\end{split}
\end{equation*}
\end{lm} 

\begin{proof}
Define the function $a_{k,\ell} : \mathbb{S}^1 \to \C$ by $a_{k,\ell}(x) = \frac{1}{2i\pi(k T'(x) - \ell) }$ and the differential operator
\begin{equation*}
\begin{split}
L_{a_{k,\ell}} : u \mapsto (a_{k,\ell} u)'.
\end{split}
\end{equation*}
Then for all $m \in \N$ we have
\begin{equation*}
\begin{split}
\langle \mathcal{L} e_{\ell}, e_k\rangle_{L^2} = \int_{\mathbb{S}^1} e^{2 i \pi (\ell x - k T(x))} \mathrm{d}x = \int_{\mathbb{S}^1} e^{2 i \pi (\ell x - k T(x))} L^m_{a_{k,\ell}}(1) (x) \mathrm{d}x,
\end{split}
\end{equation*}
so that
\begin{equation*}
\begin{split}
\va{\langle \mathcal{L} e_{\ell}, e_k\rangle_{L^2}} \leq \n{L_{a_{k,\ell}}^m(1)}_{\infty}.
\end{split}
\end{equation*}
In order to bound $L_{a_{k,\ell}}$, we first investigate the derivatives of $a_{k,\ell}$. By Faa di Bruno's formula, for all $n \in \N$ and $x \in \mathbb{S}^1$ we have
\begin{equation*}
\begin{split}
& a_{k,\ell}^{(n)} (x) = \frac{1}{2 i \pi} \sum_{m_1 + 2 m_2 + \dots + n m_n = n} (-1)^{m_1 + \dots m_n} \frac{n! (m_1 + \dots + m_n)!}{m_1! \dots m_n!} \\ & \qquad \qquad \qquad \qquad \qquad \times \frac{1}{(k T'(x) - \ell)^{1 + m_1 + \dots + m_n}} \prod_{j=1}^{n} \p{ \frac{k T^{(1 + j)}(x)}{j!}}^{m_j}.
\end{split}
\end{equation*}
Thus, since $T'$ belongs to the class $\mathcal{C}^M$,
\begin{equation*}
\begin{split}
& \va{a_{k,\ell}^{(n)} (x)} \leq \frac{1}{2 \pi} \sum_{m_1 + 2 m_2 + \dots + n m_n = n} \frac{n! (m_1 + \dots + m_n)!}{m_1! \dots m_n!} \\ & \qquad \qquad \qquad \qquad \qquad \times \frac{1}{\va{k T'(x) - \ell}^{1 + m_1 + \dots + m_n}} \prod_{j=1}^{n} \p{ \va{k} C R^j M_j}^{m_j},
\end{split}
\end{equation*}
where $C,R > 0$ are from the definition of $\mathcal{C}^M$. From the log convexity of $\p{M_j}_{j \in \N}$, it follows that if $m_1 + \dots + n m_n = n$ then (use the fact that $\p{\frac{M_j}{M_{j-1}}}_{j \geq 1}$ is increasing)
\begin{equation}\label{eqlogconv}
\begin{split}
\prod_{j=1}^n M_j^{m_j} \leq M_n.
\end{split}
\end{equation}
Notice also that $\va{k T'(x) - \ell } \geq \lambda \va{k} - \va{\ell} > \delta \va{k}$, where $\delta = \lambda - \theta$. Thus we have (assuming that $C > 1$ and $\delta < 1$, which is true without loss of generality)
\begin{equation*}
\begin{split}
& \va{a_{k,\ell}^{(n)} (x)} \leq \frac{M_n}{2 \pi \delta \va{k}}  \p{\frac{CR}{\delta}}^n \sum_{m_1 + 2 m_2 + \dots + n m_n = n} \frac{n! (m_1 + \dots + m_n)!}{m_1! \dots m_n!}.
\end{split}
\end{equation*}
Now, notice that
\begin{equation*}
\sum_{m_1 + 2 m_2 + \dots + n m_n = n} \frac{n! (m_1 + \dots + m_n)!}{m_1! \dots m_n!} = \begin{cases} 1 & \textrm{ if } n = 0 \\ 2^{n-1}n! & \textrm{ otherwise} \end{cases}.
\end{equation*}
Indeed, as a consequence of Faa di Bruno's formula, the sum in the left hand side is the $n$th derivative at zero of the function
\begin{equation*}
\begin{split}
x \mapsto 1 + \frac{x}{1-2x} = \frac{1}{1 - \frac{x}{1-x}}.
\end{split}
\end{equation*}

Notice then that for all $m \in \N$ there are natural integer coefficients that do not depend on $a_{k,\ell}$ such that 
\begin{equation}\label{eqruse}
\begin{split}
L_{a_{k,\ell}}^m 1 = \sum_{n_1 + \dots + n_m = m} c_{n_1,\dots,n_m} \prod_{j=1}^m a_{k,\ell}^{(n_j)}.
\end{split}
\end{equation}
Thus, using \eqref{eqlogconv} again,
\begin{equation*}
\begin{split}
\n{L_{a_{k,\ell}}^m(1)}_{\infty} & \leq \sum_{n_1 + \dots + n_m = m} c_{n_1,\dots,n_m} \prod_{j=1}^m \p{n_j! \frac{M_{n_j}}{2 \pi \delta \va{k} } \p{\frac{2CR}{\delta}}^{n_j}} \\
    & \leq M_m \p{\frac{CR}{\pi \delta^2 \va{k}}}^m\sum_{n_1 + \dots + n_m = m} c_{n_1,\dots,n_m} \prod_{j=1}^m n_j!.
\end{split}
\end{equation*}
Now replacing $a_{k,\ell}$ by the function $a : x \mapsto \frac{1}{1 - x}$ in \eqref{eqruse} we have that (notice that $L_a^m (1) :x \mapsto \frac{(2m)!}{m! 2^m} \frac{1}{(1-x)^{2m}}$, where $L_a$ is the differential operator defined by $L_a(u) = (au)'$)
\begin{equation*}
\begin{split}
\sum_{n_1 + \dots + n_m = m} c_{n_1,\dots,n_m} \prod_{j=1}^m n_j! = L_a^m(1)(0) = \frac{(2m)!}{m! 2^m}.
\end{split}
\end{equation*}
Thus
\begin{equation*}
\begin{split}
\n{L_{a_{k,\ell}}^m(1)}_{\infty} \leq m! M_m \frac{(2m)!}{m!^2} \p{\frac{CR}{2 \pi \delta^2 \va{k}}}^m.
\end{split}
\end{equation*}
We only need to notice that $\frac{(2m)!}{m!^2}$ grows at most exponentially to end the proof (with different values of $C$ and $R$ of course).
\end{proof}

We can now prove Lemma \ref{lmmixedterm}.

\begin{proof}[Proof of Lemma \ref{lmmixedterm}]
We will only deal with the case $n \neq 0$, the case $n=0$ is similar. Let us compute (here $\mathcal{L}^*$ denotes the $L^2$-adjoint of $\mathcal{L}$, that is the Koopman operator):
\begin{equation*}
\begin{split}
\n{\pi_m \mathcal{L} \pi_n u}_{L^2}^2 & = \sum_{\theta^{m-1} \leq \va{k} < \theta^{m}} \va{\langle \mathcal{L} \pi_n u, e_k \rangle_{L^2}}^2 =  \sum_{\theta^{m-1} \leq \va{k} < \theta^{m}} \va{\langle \pi_n u, \mathcal{L}^* e_k \rangle_{L^2}}^2 \\
    & \leq \n{\pi_n u}_{L^2}^2 \sum_{\theta^{m-1} \leq \va{k} < \theta^{m}} \n{\pi_n \mathcal{L}^* e_k}_{L^2}^2 \\
    & \leq \n{\pi_n u}_{L^2}^2 \sum_{\theta^{m-1} \leq \va{k} < \theta^{m}} \sum_{\theta^{n-1} \leq \va{\ell} < \theta^{n}} \va{\langle \mathcal{L} e_\ell, e_k \rangle_{L^2}}^2.
\end{split}
\end{equation*}
Now, if $\theta^{m-1} \leq \va{k} < \theta^{m}$ and $\theta^{n-1} \leq \va{\ell} < \theta^{n}$ then we have
\begin{equation*}
\begin{split}
\va{k} \geq \theta^{m-1} \geq \theta^{n-1} > \theta^{-1} \va{\ell}
\end{split}
\end{equation*}
and thus by Lemma \ref{lmfonda} we have (recall that $w$ is increasing)
\begin{equation*}
\begin{split}
\va{\langle \mathcal{L} e_\ell, e_k \rangle_{L^2}} \leq C w\p{\frac{R}{\va{k}}} \leq C w\p{\frac{\theta R}{\theta^m}}.
\end{split}
\end{equation*}
Consequently,
\begin{equation*}
\begin{split}
\n{\pi_m \mathcal{L} \pi_n u}_{L^2}^2 & \leq 4 C^2 \n{\pi_n u}^2_{L^2} \p{\theta^{m} - \theta^{m-1} +1} \p{\theta^{n} - \theta^{n-1} +1} w\p{\frac{\theta R}{\theta^m}}^2
\end{split}
\end{equation*}
and the result follows.
\end{proof}

We will need another technical result to prove Theorem \ref{thm1quant}. For all $N \in \N$, define the following finite rank operators on $\mathcal{H}$:
\begin{equation}\label{eqfiniterankoperators}
\begin{split}
A_N = \sum_{0 \leq n \leq m \leq N} \pi_m \mathcal{L} \pi_n \textrm{ and } B_N =\sum_{0 \leq m < n \leq N} \pi_m \mathcal{L} \pi_n.
\end{split}
\end{equation}
We will use these finite rank operators to approximate the transfer operator $\mathcal{L}$, to do so we need the following lemma.

\begin{lm}\label{lmcauchy}
There is a constant $C > 0$ such that for all $M \geq N \geq 0$ we have
\begin{equation*}
\begin{split}
\n{A_N - A_M}_{L^2 \to \mathcal{H}} \leq C \sup_{m > N} \frac{w\p{\frac{R}{\theta^{m}}}}{w\p{\frac{R}{\theta^{m-1}}}}
\end{split}
\end{equation*}
and
\begin{equation*}
\begin{split}
\n{B_N - B_M}_{\mathcal{H} \to \mathcal{H}} \leq C \sup_{m \geq N} \frac{w\p{\frac{R}{\theta^{m}}}}{w\p{\frac{R}{\theta^{m-1}}}}.
\end{split}
\end{equation*}
\end{lm}

\begin{proof}
If $u \in \mathcal{H}$ then we have
\begin{equation*}
\begin{split}
\p{A_M - A_N}u = \sum_{\substack{ 0 \leq n \leq m \leq M \\ N < m}} \pi_m \mathcal{L} \pi_n u
\end{split}
\end{equation*}
and thus
\begin{equation*}
\begin{split}
\n{\p{A_M - A_N}u}_{\mathcal{H}}^2 & = \sum_{N < m \leq M} \lambda^{-2m} w\p{\frac{R}{\theta^{m-1}}}^{-2} \n{\pi_m \mathcal{L} \sum_{n \leq m} \pi_n u}_{L^2}^2.
\end{split}
\end{equation*}
But if $N < m \leq M$ we have with Lemma \ref{lmmixedterm}
\begin{equation*}
\begin{split}
\n{\pi_m \mathcal{L} \sum_{n \leq m} \pi_n u}_{L^2} & \leq \sum_{n \leq m} \n{\pi_m \mathcal{L} \pi_n u}_{L^2} \\
    & \leq C w\p{\frac{R}{\theta^m}} \theta^{\frac{m}{2}} \sum_{n \leq m}\theta^{\frac{n}{2}} \n{\pi_n u}_{L^2} \\
    & \leq \widetilde{C} w\p{\frac{R}{\theta^m}} \theta^{m} \sqrt{\sum_{n \leq m} \n{\pi_n u}_{L^2}^2},
\end{split}
\end{equation*}
and thus (for some new constant $C$ that may change from one line to another)
\begin{equation*}
\begin{split}
\n{\p{A_M - A_N}u}_{\mathcal{H}}^2 & \leq C \sup_{m > N} \p{\frac{w\p{\frac{R}{\theta^{m}}}}{w\p{\frac{R}{\theta^{m-1}}}}}^2 \sum_{n \geq 0} \p{\sum_{m  > N} \frac{\theta^{2m}}{\lambda^{2m}}} \n{\pi_n u}_{L^2}^2 \\
   & \leq C \sup_{m > N} \p{\frac{w\p{\frac{R}{\theta^{m}}}}{w\p{\frac{R}{\theta^{m-1}}}}}^2 \sum_{n \geq 0} \n{\pi_n u}_{L^2}^2 \\
   & \leq C \sup_{m > N} \p{\frac{w\p{\frac{R}{\theta^{m}}}}{w\p{\frac{R}{\theta^{m-1}}}}}^2 \n{u}_{L^2}^2.
\end{split}
\end{equation*}

Before proving the second estimate, let us show that there is a constant $C > 0$ such that for every integer $n$ we have
\begin{equation}\label{eq:unpeudedetail}
\begin{split}
\sum_{0 \leq m < n} \lambda^{-2m} w \p{\frac{R}{\theta^{m-1}}}^{-2} \leq C \lambda^{-2n} w\p{\frac{R}{\theta^{n-2}}}^{-2}.
\end{split}
\end{equation}
To do so, recall the function $k(x)$ from the proof of Lemma \ref{lmpropw} and choose $m_0$ large enough so that 
\begin{equation*}
\begin{split}
\frac{\lambda^2}{\theta^{2k\p{\frac{R}{\theta^{m_0 - 1}}}}} < 1.
\end{split}
\end{equation*}
Then, when $n$ is large enough, we may split the sum in \eqref{eq:unpeudedetail} between the sum over $0 \leq m < m_0$ and the sum over $m_0 \leq m < n$. The first sum is independent on $n$, and can consequently be ingored since the right hand side of \eqref{eq:unpeudedetail} tends to $+ \infty$ when $n$ tends to $+ \infty$, according to Lemma \ref{lmpropw}. To bound the second sum, recall \eqref{eqdecayw} to see that
\begin{equation*}
\begin{split}
& \lambda^{2n} w\p{\frac{R}{\theta^{n-1}}}^2 \sum_{m_0 \leq m < n} \lambda^{-2m} w \p{\frac{R}{\theta^{m-1}}}^{-2} \\ & \qquad \qquad \qquad \qquad \leq \sum_{m_0 \leq m < n} \lambda^{2(n-m)} \p{\theta^{m-n}}^{2k\p{\frac{R}{\theta^{m-1}}}} \\
    & \qquad \qquad \qquad \qquad \leq \sum_{\ell \geq 0} \p{\frac{\lambda^2}{\theta^{2k\p{\frac{R}{\theta^{m_0} - 1}}}}}^{\ell} < + \infty.
\end{split}
\end{equation*}

We turn now to the proof of the second estimate and write for $u \in \mathcal{H}$
\begin{equation*}
\begin{split}
\p{B_M - B_N}u = \sum_{\substack{ 0 \leq m < n \leq M \\ N <n}} \pi_m \mathcal{L} \pi_n u
\end{split}
\end{equation*}
from which we get (we use \eqref{eq:unpeudedetail} on the fifth line and $C$ may change from one line to another)
\begin{align*}
\n{\p{B_M - B_N}u}_{\mathcal{H}}^2 & = \sum_{0 \leq m < M} \lambda^{-2m} w\p{\frac{R}{\theta^{m-1}}}^{-2} \n{\pi_m \mathcal{L} \sum_{n > \max(m,N)} \pi_n u}_{L^2}^2 \\
      & \leq C \sum_{0 \leq m < M} \lambda^{-2m} w\p{\frac{R}{\theta^{m-1}}}^{-2} \n{ \sum_{n > \max(m,N)} \pi_n u}_{L^2}^2 \\
      & \leq C \sum_{0 \leq m < M} \lambda^{-2m} w\p{\frac{R}{\theta^{m-1}}}^{-2} \sum_{n > \max(m,N)} \n{ \pi_n u}_{L^2}^2 \\
      & \leq C \sum_{n > N}  \p{\sum_{0 \leq m < n} \lambda^{-2m} w\p{\frac{R}{\theta^{m-1}}}^{-2}} \n{\pi_n u}_{L^2}^2 \\
      & \leq C \sum_{n > N} \lambda^{-2n} w\p{\frac{R}{\theta^{n-2}}}^{-2} \n{\pi_n u}_{L^2}^2 \\
      & \leq C\sup_{n > N} \p{\frac{w\p{\frac{R}{\theta^{n-1}}}}{w\p{\frac{R}{\theta^{n-2}}}}}^2 \n{u}_{\mathcal{H}}^2. 
\end{align*}
\end{proof}

We are now in position to end the proof of Theorem \ref{thm1quant}.

\begin{proof}[Proof of Theorem \ref{thm1quant}]
Lemma \ref{lmcauchy} implies in particular that the sequence $\p{A_N}_{N \in \N}$ is a Cauchy sequence of bounded operators from $L^2$ to $\mathcal{H}$ and thus converges to a bounded operator $\mathcal{L}_c : L^2 \to \mathcal{H}$. For the same reason, $\p{B_N}_{N \in \N}$ converges to a bounded operator $\mathcal{L}_b : \mathcal{H} \to \mathcal{H}$. By checking the identity on trigonometric polynomials, we see that
\begin{equation}\label{eqnpd}
\begin{split}
\mathcal{L} = \mathcal{L}_c + \mathcal{L}_b.
\end{split}
\end{equation}
In particular, $\mathcal{L}$ is bounded (and even compact, as a limit of finite rank operators) from $\mathcal{H}$ to itself.

The only thing that we still need to check is the bound \eqref{eqboundsing} on singular values of the operator $\mathcal{L}$ acting on $\mathcal{H}$. If $N \in \N$, notice that the operator $A_N + B_N$ has rank at most $2 \lceil \theta^{N} \rceil + 1$ (where $A_N$ and $B_N$ are defined by \eqref{eqfiniterankoperators}). From Lemma \ref{lmcauchy} (letting $M$ tend to infinity), we deduce that
\begin{equation*}
\begin{split}
\n{\mathcal{L} - \p{A_N+ B_N}}_{\mathcal{H} \to \mathcal{H}} \leq 2 C \sup_{m \geq N} \frac{w\p{\frac{R}{\theta^m}}}{w\p{\frac{R}{\theta^{m-1}}}}
\end{split}
\end{equation*}
and thus (see \cite[Theorem IV.2.5]{Gohb})
\begin{equation*}
\begin{split}
\sigma_{2 \lceil \theta^{N} \rceil + 2} \leq 2 C \sup_{m \geq N} \frac{w\p{\frac{R}{\theta^m}}}{w\p{\frac{R}{\theta^{m-1}}}}.
\end{split}
\end{equation*}
The result then follows from the fact that the sequence $\p{\sigma_k}_{k \in \N}$ is decreasing.
\end{proof}

\section{Nuclear power decomposition}\label{secnpd}

We saw in the proof of Theorem \ref{thm1quant} that the transfer operator $\mathcal{L}$ may be written as the sum \eqref{eqnpd} of the operators $\mathcal{L}_b$ and $\mathcal{
L}_c$. In this section, we show that this is a nuclear power decomposition in the spirit of \cite{Tsu}, and we investigate the consequences of the existence of such a decomposition, in particular in terms of dynamical determinants (see Propositions \ref{propfact} and \ref{propbound}). Thus, we will prove in particular Theorem \ref{thmintro2}.

We first investigate the operator $\mathcal{L}_b$. To do so, define the function $g : \N \mapsto \R_+^*$ by
\begin{equation}\label{eqdefg}
\begin{split}
g(N) = \sup_{m \geq N} \frac{w\p{\frac{R}{\theta^{m}}}}{w\p{\frac{R}{\theta^{m-1}}}}
\end{split}
\end{equation}
and notice that $g(N) \underset{N \to + \infty}{\to} 0$ by Lemma \ref{lmpropw}. The operator $\mathcal{L}_b$ is morally strictly upper triangular, the following proposition uses the function $g$ to quantify the fact that $\mathcal{L}_b$ is morally nilpotent.

\begin{prop}\label{propfastdecay}
There is a constant $C > 0$ such that for all $n \in \N^*$ we have
\begin{equation*}
\begin{split}
\n{\mathcal{L}_b^n}_{\mathcal{H} \to \mathcal{H}} \leq C^n \prod_{k=0}^{n-1} g(k).
\end{split}
\end{equation*}
In particular, the spectral radius of $\mathcal{L}_b$ is zero (i.e. $\mathcal{L}_b$ is quasi-nilpotent).
\end{prop}

\begin{proof}
Notice that if $k < N$ then from the definition of $B_N$ it comes that
\begin{equation*}
\begin{split}
\p{B_N -B_k} B_N = \p{B_N - B_k}\p{B_N - B_{k+1}}.
\end{split}
\end{equation*}
Thus if $N \geq n-1$ we have
\begin{equation*}
\begin{split}
B_N^n = \prod_{k=0}^{n-1}\p{B_N - B_k}.
\end{split}
\end{equation*}
Letting $N$ tends to infinity, we get that
\begin{equation*}
\begin{split}
\mathcal{L}_b^n = \prod_{k=0}^{n-1} \p{\mathcal{L}_b - B_k}
\end{split}
\end{equation*}
and the result follows from Lemma \ref{lmcauchy}.
\end{proof}

Then we investigate the operator $\mathcal{L}_c$ (as an operator from $\mathcal{H}$ to itself).

\begin{prop}\label{propnucpart}
There are constants $C,R' > 0$ such that, if we define the function $f$ for $x > 0$ by
\begin{equation}\label{eqdeff}
\begin{split}
f(x) = x^\alpha w\p{\frac{R'}{x}},
\end{split}
\end{equation}
where $\alpha = \frac{\log \lambda}{\log \theta}$, and if $\p{s_\ell}_{\ell \in \N}$ denotes the sequence of singular values of $\mathcal{L}_c$ acting on $\mathcal{H}$, then, for all $\ell \geq 1$, we have
\begin{equation}\label{eqavecunsupmieux}
\begin{split}
s_\ell \leq C f(\ell).
\end{split}
\end{equation}
In particular, $\mathcal{L}_c$ is nuclear of order $0$.
\end{prop}

\begin{proof}
Since $\mathcal{L}_c$ is continuous from $L^2$ to $\mathcal{H}$, we have the following bound on its singular values as a compact operator from $\mathcal{H}$ to itself:
\begin{equation*}
\begin{split}
\forall m \in \N: s_{2 \lceil \theta^{m} \rceil + 1} \leq C \sup_{p \geq m} \lambda^{p} w\p{\frac{R}{\theta^{p-1}}}
\end{split}
\end{equation*}
for some constant $C > 0$. Since $f$ does not vanish, we only need to prove \eqref{eqavecunsupmieux} for $\ell$ large. Thus, let $\ell$ be large and let $m$ be the largest integer such that $\ell \geq 2 \lceil \theta^{m} \rceil +1$. Then, we have
\begin{equation}\label{eqdetail}
\begin{split}
s_\ell \leq s_{2 \lceil \theta^{m} \rceil + 1} \leq C \sup_{y \geq \theta^m} y^\alpha w\p{\frac{\theta R}{y}}.
\end{split}
\end{equation}
Here, we performed the change of variables ``$y = \theta^p$". Then, notice that $\theta^m \geq \frac{1}{2\theta} \ell - \frac{3}{2 \theta} \geq \frac{1}{4 \theta}\ell$ (provided that $\ell$ is large enough). Hence, we deduce from \eqref{eqdetail} that, for $\ell$ large enough, we have (with the change of variables ``$x = 4 \theta y$'' and taking $R' = 4 \theta^2 R$ in the definition of $f$)
\begin{equation}\label{eqpresque}
\begin{split}
s_\ell \leq C \sup_{x \geq \ell} f(x).
\end{split}
\end{equation}
Recall the function $k(x)$ from the proof of Lemma \ref{lmpropw} and use \eqref{eqdecayw} to see that for all $x \geq \ell$ we have
\begin{equation}
\begin{split}
\frac{f(x)}{f(\ell)} \leq \p{\frac{x}{\ell}}^{\alpha - k\p{\frac{R'}{\ell}}},
\end{split}
\end{equation}
but if $\ell$ is large enough we have $k\p{\frac{R'}{\ell}} > \alpha$ and consequently $f(x) \leq f(\ell)$. Hence, for $\ell$ large enough we have
\begin{equation}
\begin{split}
f(\ell) = \sup_{x \geq \ell} f(x),
\end{split}
\end{equation}
and \eqref{eqavecunsupmieux} follows from \eqref{eqpresque}. To see that $\mathcal{L}_c$ is nuclear of order $0$, recall from Lemma \ref{lmpropw} that $w$ vanishes at all orders in $0$. Hence, $f$ decays faster than any polynomial and so does the sequence of singular values of $\mathcal{L}_c$.
\end{proof}

Now, following \cite{Tsu}, we want to use the nuclear power decomposition in order to study the dynamical determinant $d$ defined by \eqref{eqdefdet}. This is the point of Proposition \ref{propfact}, that completes the proof of Theorem \ref{thmintro2}.

\begin{prop}\label{propfact}
If $z$ is small enough then we have
\begin{equation}\label{eqtemp}
\begin{split}
d(z) = \det\p{I - z\p{I - z\mathcal{L}_b}^{-1} \mathcal{L}_c }.
\end{split}
\end{equation}
In particular, $d$ has a holomorphic extension to $\C$ whose zeroes are exactly the inverses of the non-zero eigenvalues of $\mathcal{L}$ acting on $\mathcal{H}$ (counted with multiplicity).
\end{prop}

Notice in particular that this proposition implies that the spectrum of $\mathcal{L}$ acting on $\mathcal{H}$ coincides with the Ruelle spectrum of $\mathcal{L}$ defined in the first part of \cite{Bal2} (the inverse of the zeroes of $d$ are the Ruelle resonances according to \cite[Theorem 3.3]{Bal2}).

\begin{proof}[Proof of Proposition \ref{propfact}]
If $N \in \N$ then the operators $B_N$ and $\mathcal{L}_c$ are trace class (recall that $B_N$ is defined by \eqref{eqfiniterankoperators}). Moreover, $B_N$ is nilpotent and thus
\begin{equation}\label{eq:convergence_determinant}
\begin{split}
\det\p{I - z \p{B_N + \mathcal{L}_c}} & = \det\p{I - z B_N} \det \p{I - z \p{I - z B_N}^{-1} \mathcal{L}_c} \\ & = \det \p{I - z \p{I - z B_N}^{-1} \mathcal{L}_c} \\ & \underset{N \to + \infty}{\to} \det\p{I - z\p{I - z \mathcal{L}_b}^{-1} \mathcal{L}_c},
\end{split}
\end{equation}
and the convergence holds uniformly on every compact subset of $\C$. Denote by $h(z)$ the entire function on the right hand side of \eqref{eqtemp}. Since $h(0) = 1$, there is a sequence $\p{a_n}_{n \geq 1}$ of complex numbers such that for $\va{z}$ small enough we have
\begin{equation*}
\begin{split}
h(z) = \exp\p{- \sum_{n \geq 1} \frac{a_n}{n} z^n}.
\end{split}
\end{equation*}
Then applying Cauchy's formula, we find for $n \geq 1$ (and for $\epsilon$ small enough)
\begin{equation*}
\begin{split}
a_n = - \frac{1}{2 i \pi} \int_{\partial \mathbb{D}\p{0,\epsilon}} \frac{h'(z)}{h(z)} \frac{\mathrm{d}z}{z^n} = \lim_{N \to + \infty} \textrm{tr}\p{\p{B_N + \mathcal{L}_c}^n}.
\end{split}
\end{equation*}
Then notice that, since $B_N$ is nilpotent we have 
\begin{equation*}
\textup{tr}\p{(B_N + \mathcal{L}_c)^n} = \textup{tr}\p{(B_N + \mathcal{L}_c)^n - B_N^n}
\end{equation*}
but the operators $(B_N + \mathcal{L}_c)^n - B_N^n$ converge in trace class topology to the operator $\mathcal{L}^n - \mathcal{L}_b^n$. Thus we have
\begin{equation}\label{eqcalcultrace}
\begin{split}
a_n & = \textup{tr}\p{\mathcal{L}^n - \mathcal{L}_b^n} = \sum_{k \in \Z} \langle \p{\mathcal{L}^n - \mathcal{L}_b^n} e_k, e_k \rangle_{L^2}. 
\end{split}
\end{equation}
Now, notice that if $k \in \Z$ and $n \in \N^*$ we have $\langle \mathcal{L}_b^n e_k, e_k\rangle_{L^2} = 0$. Indeed, if $k \in \Z^*$ is such that $\theta^{m-1} \leq \va{k} < \theta^m$, then for every $N \in \N$, the image of $e_k$ by $B_N^n$ belongs to the span of the $e_{\ell}$'s such that $\va{\ell} < \theta^{m-n}$. In particular, $B_N^n e_k$ is orthogonal to $e_k$. Since, we also have $B_N^n e_0 = 0$, we find that, for every $k \in \Z$, we have $\langle B_N^n e_k, e_k\rangle_{L^2} = 0$, and, letting $N$ tends to infinity, that $\langle \mathcal{L}_b^n e_k, e_k\rangle_{L^2} = 0$. Hence, \eqref{eqcalcultrace} gives
\begin{equation}\label{eqcalcultrace2}
\begin{split}
    a_n & = \sum_{k \in \Z} \langle \mathcal{L}^n e_k, e_k \rangle_{L^2}  = \sum_{k \in \Z} \int_{\mathbb{S}^1} e^{2ik\pi \p{x - T^n(x)}} \mathrm{d}x \\
        & = \lim_{m \to + \infty} \int_{\mathbb{S}^1} \frac{\sin\p{(2m+1) \pi (x - T^n(x))}}{\sin\p{\pi\p{x-T^n(x)}}} \mathrm{d}x.
\end{split}
\end{equation}
Finally, we use a partition of unity in the last integral and locally we perform the change of variable ``$u= x -T^n(x)$''. Then, we recognize the Dirichlet kernel and find $a_n = \textup{tr}^\flat\p{\mathcal{L}^n}$.

The fact that the zeros of $d$ are exactly the inverses of the non-zero eigenvalues of $\mathcal{L}$ counted with multiplicity follows from \cite[Theorem 3.1]{kur}. However, notice that in our case the situation is simpler than in the general theory of Weinstein--Aronszajn determinant, and that the correspondence between the zeros of $d$ and the inverses of the non-zero eigenvalues of $\mathcal{L}$ may be deduced from the convergence \eqref{eq:convergence_determinant}.
\end{proof}

In some cases, it may happen that $\mathcal{L}$ acting on $\mathcal{H}$ is trace class, or in some Schatten class. In these cases, we may simplify Proposition \ref{propfact} in the following way.

\begin{prop}\label{propordre}
Assume that there is $p > 0$ such that $\mathcal{L}$ acting on $\mathcal{H}$ is in the Schatten class $\mathcal{S}_p$. Then, if $m$ denotes the smallest integer larger than $p$, we have
\begin{equation}\label{eqdetm}
\begin{split}
d(z) = \textup{det}_m\p{I - z \mathcal{L}} \exp\p{- \sum_{n=1}^{m-1} \frac{\textup{tr}^\flat\p{\mathcal{L}^n}}{n} z^n},
\end{split}
\end{equation}
where the $\textup{tr}^\flat \p{\mathcal{L}^n}$ are defined by \eqref{eqdeftrace} and $\textup{det}_m$ denotes the regularized determinant of order $m$ defined in \cite[IX]{Gohb} (this is the usual Fredholm determinant when $m=1$). In particular, the order of $d$ is less than $p$.
\end{prop}

\begin{proof}
We know that when $\va{z}$ is small enough we have
\begin{equation*}
\begin{split}
\textup{det}_m\p{I - z \mathcal{L}} = \exp\p{- \sum_{n \geq m} \frac{\textup{tr}\p{\mathcal{L}^n}}{n}z^n}.
\end{split}
\end{equation*}
Then, the same computation \eqref{eqcalcultrace}-\eqref{eqcalcultrace2} as in the proof of Proposition \ref{propfact} ensures that for $n \geq m$ we have
\begin{equation}
\begin{split}
\textup{tr}\p{\mathcal{L}^n} = \sum_{k \in \Z} \langle \mathcal{L}^n e_k, e_k \rangle_{L^2} = \textup{tr}^\flat\p{\mathcal{L}^n},
\end{split}
\end{equation}
and \eqref{eqdetm} follows. To see that $d$ has order less than $p$, recall that the since $\mathcal{L}$ belongs to the Schatten class $\mathcal{S}_p$, its eigenvalues are $p$-summable, then apply Lidskii's Trace Theorem to recognize that $\textup{det}_m\p{I - z \mathcal{L}}$ is a Weierstrass product and hence has order less than $p$ thanks to \cite[Theorem 2.6.5]{Boas}.
\end{proof}

\begin{rmq}
If the sequence $M$ satisfies \eqref{eqgentil} for some $\gamma > 0$, then the estimates \eqref{eqpoly} and \eqref{eqboundsing}, respectively from Lemma \ref{lmpropw} and Theorem \ref{thm1quant} imply that the singular values of $\mathcal{L}$ acting on $\mathcal{H}$ satisfy
\begin{equation}
\begin{split}
\sigma_k \underset{k \to + \infty}{=} \mathcal{O}\p{\frac{1}{k^\delta}},
\end{split}
\end{equation}
where $\delta = \frac{\log \theta}{\log \gamma}$. Hence, $\mathcal{L}$ acting on $\mathcal{H}$ belongs to the Schatten class $\mathcal{S}_p$ for any $p > \delta^{-1}$ and Proposition \ref{propordre} implies that the order of $d$ is less than $\frac{\log \gamma}{\log \theta}$. Since $\theta$ may be chosen arbitrarily close to the expanding constant $\lambda$, the following result follows.
\end{rmq}

\begin{cor}\label{corordre}
If there is $\gamma > 0$ such that the sequence $M$ satisfies \eqref{eqgentil}, then the order of the dynamical determinant $d$ is less than $\frac{\log \gamma}{\log \lambda}$.
\end{cor}

\begin{rmq}\label{rmqordre}
Notice that \eqref{eqgentil} implies that taking a derivative in the class $\mathcal{C}^M$ results in replacing $R$ by $\gamma R$ in \eqref{eqdefdc}. Composing by a contraction of factor $\lambda^{-1}$ (which is basically what $\mathcal{L}$ does) results in replacing $R$ by $\lambda^{-1} R$. Thus $\mathcal{L}$ has morally the same regularizing effect in the class $\mathcal{C}^M$ as taking $\frac{\log \gamma}{\log \lambda}$ primitives (notice that this number is not necessarily an integer). Consequently, the decay that we obtain on the singular values of $\mathcal{L}$, and ultimately the bound on the order of the dynamical determinant, is natural (considering for instance the case of Sobolev injections).
\end{rmq}

Finally, we will use the nuclear power decomposition \eqref{eqnpd} with Propositions \ref{propfastdecay}, \ref{propnucpart} and \ref{propfact} in order to bound the growth of the dynamical determinant $d$. To do so, define the entire functions $F$ and $G$ by
\begin{equation}\label{eqdefFetG}
\begin{split}
F(z) = \p{1 +z}\prod_{m = 1}^{+ \infty} \p{1 + f(m)z} \textrm{ and } G(z) = \sum_{n = 0}^{+ \infty} \p{\prod_{k=0}^{n-1} g(k)} z^n,
\end{split}
\end{equation}
where we recall that $f$ and $g$ have been defined respectively in \eqref{eqdeff} and \eqref{eqdefg}. Notice that $F$ has genus zero. Thus if $n(r)$ denotes the number of integers $m$ such that $f(m)^{-1} \leq r$ or $m =0$, we have the following estimate \cite[Lemma 3.5.1]{Boas} for $r > 0$:
\begin{equation}\label{eqwk}
\begin{split}
\log F(r) \leq \int_0^r \frac{n(s)}{s} \mathrm{d}s + r \int_r^{+ \infty} \frac{n(s)}{s^2} \mathrm{d}s.
\end{split}
\end{equation}
This bound may be used with Proposition \ref{propbound} to control the growth of the dynamical determinant, see \S \ref{secex} for examples.

\begin{prop}\label{propbound}
There is a constant $C$ such that for all $z \in \C$ we have
\begin{equation*}
\begin{split}
\va{d(z)} \leq F\p{C \va{z} G\p{C\va{z}}}.
\end{split}
\end{equation*}
\end{prop}

\begin{proof}
Let $z \in \C$. Denote by $\p{c_k}_{k \in \N}$ the sequence of singular values of the operator $-z\p{I - z\mathcal{L}_b}^{-1} \mathcal{L}_c$ and by $\p{\lambda_k}_{k \in \N}$ the sequence of its eigenvalues. By Lidskii's Theorem we have
\begin{equation*}
\begin{split}
\va{d(z)} = \va{\prod_{k \in \N} \p{1 + \lambda_k}} \leq 1 + \sum_{n \geq 1} \sum_{k_1 < \dots < k_n} \prod_{j=1}^n \va{\lambda_{k_j}}.
\end{split}
\end{equation*}
Then applying \cite[Theorem IV.3.1]{Gohb} we see that 
\begin{equation*}
\begin{split}
\va{d(z)} \leq 1 + \sum_{n \geq 1} \sum_{k_1 < \dots < k_n} \prod_{j=1}^n c_{k_j} = \prod_{k \geq 0} \p{1 + c_{k}}.
\end{split}
\end{equation*}
Now if $\p{s_k}_{k \geq 0}$ denotes the sequence of singular values of $\mathcal{L}_c$ then we have for $k \geq 1$ (replace $f(k)$ by $1$ in the case $k = 0$)
\begin{equation*}
\begin{split}
c_k \leq \va{z} \n{\p{I - z \mathcal{L}_b}^{-1}}_{\mathcal{H} \to \mathcal{H}} s_k \leq C \va{z} \n{\p{I - z \mathcal{L}_b}^{-1}}_{\mathcal{H} \to \mathcal{H}} f(k),
\end{split}
\end{equation*}
for some constant $C > 0$, and thus $\va{d(z)} \leq F\p{C \va{z} \n{\p{I - z \mathcal{L}_b}^{-1}}_{\mathcal{H} \to \mathcal{H}}}$. But from Lemma \ref{propfastdecay}, we get that, up to taking larger $C$, we have 
\begin{equation*}
\n{\p{I - z \mathcal{L}_b}^{-1}}_{\mathcal{H} \to \mathcal{H}} \leq G\p{C \va{z}},
\end{equation*}
which ends the proof of the proposition.
\end{proof}

\begin{rmq}\label{rmqgenuszero}
Assume that the right hand side of \eqref{eqboundsing} in Theorem \ref{thm1quant} is summable. Then we know that $\mathcal{L}$ acting on $\mathcal{H}$ is trace class and Proposition \ref{propordre} implies that the order of the dynamical determinant $d$ is less than $1$. In particular, $d$ has genus zero, and when the right hand side of \eqref{eqboundsing} decays very fast we may want to get a better bound on $d$. To do so, we may work as in the proof of Proposition \ref{propbound} to find that there is a constant $C > 0$ such that for every $z \in \C$ we have
\begin{equation}\label{eqbornedgenuszero}
\begin{split}
\va{d(z)} \leq \p{1 + C\va{z}} \prod_{k \in \N^*} \p{1 + C \va{z}  \sup_{x \leq \frac{1}{k}} \frac{w\p{Ax}}{w\p{\theta A x}}}.
\end{split}
\end{equation}
The infinite product in the right hand side of \eqref{eqbornedgenuszero} may be bounded using \cite[Lemma 3.5.1]{Boas} as we did for $F$: just replace $n(s)$ by the number of integer $k$ such that $\sup_{x \leq \frac{1}{k}} \frac{w\p{Ax}}{w\p{\theta A x}} \geq s^{-1}$ in \eqref{eqwk} or $k=0$.
\end{rmq}

\begin{rmq}
Notice that, using Jensen's formula \cite[1.2.1 p.2]{Boas}, a bound on the growth of the dynamical determinant immediately gives an upper bound on the asymptotics of the number of Ruelle resonances outside of $\mathbb{D}(0,\epsilon)$, when $\epsilon$ tends to $0$.
\end{rmq}

\section{Examples}\label{secex}

\subsection{Gevrey and analytic dynamics}\label{subsecga}

In this section we take $M_k = k!^{\sigma - 1}$ for some $\sigma \geq 1$. For $\sigma = 1$, the class $\mathcal{C}^M$ is the class of real-analytic functions. For $\sigma >1$, this is by definition the class of $\sigma$-Gevrey functions. We still denote by $T$ an expanding map of the circle with expanding factor at least $\lambda > 1$, and we assume that $T'$ is $\sigma$-Gevrey. In this case, we see that for every $\gamma > 1$, we can find $C > 0$ such that for all $k \in \N$ the estimate \eqref{eqgentil} holds. Thus the dynamical determinant $d$ has order $0$. But we can of course get a better bound.

To do so, recall the function $k$ from the proof of Lemma \ref{lmpropw}. Its definition implies that if $x > 0$ then
\begin{equation*}
\begin{split}
x^{k(x)} k(x)!^\sigma < x^{k(x) +1} \p{k(x) + 1}!^\sigma
\end{split}
\end{equation*}
and thus
\begin{equation*}
\begin{split}
k(x) > x^{-\frac{1}{\sigma}} - 1.
\end{split}
\end{equation*}

Then if $\theta \in \left]0,\lambda\right[$ and $A > 0$ is the constant from Theorem \ref{thm1quant}, we have when $m \geq 1$
\begin{equation}\label{eqestimeeGevrey}
\begin{split}
\sigma_m \leq C\sup_{0 < x \leq \frac{1}{m}} \frac{w(Ax)}{w(\theta A x)} \leq \sup_{0 < x \leq \frac{1}{m}} \p{\frac{1}{\theta}}^{k(\theta A x)} \leq c \exp\p{-c^{-1} m^{\frac{1}{\sigma}}}
\end{split}
\end{equation}
for some constant $c > 0$. Thus, by \cite[Lemma 1.13]{lagtf} (or Remark \ref{rmqgenuszero}), we have that for some constant $c > 0$ we have
\begin{equation*}
\begin{split}
\va{d(z)} \leq c \exp\p{c \p{\log_+ \va{z}}^{1 + \sigma}}.
\end{split}
\end{equation*}
Notice that we retrieve the optimal result \cite{bandtlowLowerBoundsRuelle2019} when $\sigma = 1$, which is also the result of Ruelle \cite{Ruelle1}. When $\sigma > 1$, we get a better result than \cite{lagtf}.  This is because our space is more carefully designed, in particular our estimates in Lemma \ref{lmfonda} is sharper than the one from \cite[Lemma 6.7]{lagtf}. It is likely that we could use the techniques presented here to achieve similar bounds in the context of \cite{lagtf}. In particular, using Paley--Littlewood decomposition with annuli of polynomial size may not be such a good idea. It seems easier to use a Paley--Littlewood decomposition with annuli of exponential size, with a ratio adapted to the hyperbolicity of our map (as we did here). Maybe, it would also be wise to use the characterization of singular values as approximation numbers. However, the geometrical context of \cite{lagtf} being more intricate, there are many technical points to check, and this would certainly result in a cumbersome proof (in particular, dealing with the transition from the stable direction to the unstable one requires some care).

\begin{rmq}
The bound \eqref{eqestimeeGevrey} on the singular values of $\mathcal{L}$ acting on $\mathcal{H}$ implies that the transfer operator $\mathcal{L}$ belongs to the exponential class of type $(c^{-1},\sigma^{-1})$ defined in \cite{bandtlowResolventEstimatesOperators2008}. Hence, we may apply the results from \cite{bandtlowResolventEstimatesOperators2008} to transfer operators associated to Gevrey expanding maps of the circle. For instance, the resolvent estimates \cite[Theorem 3.13]{bandtlowResolventEstimatesOperators2008} may be used to derive a better (super-exponential) remainder in the asymptotics expansion for the correlations of Gevrey observables (see \cite[Theorem 1.2]{GouLiv2} for the usual asymptotics of correlation in the case of hyperbolic diffeomorphisms). We could probably also use \cite[Theorem 4.2]{bandtlowResolventEstimatesOperators2008} to control globally the Ruelle spectrum of a perturbation of $T$ in the Gevrey category.
\end{rmq}

\subsection{The class $\mathcal{C}^{\alpha,\beta}$}

We investigate now the classes that we used in \cite{jezequelGlobalTraceFormula2019} (where they were called $\mathcal{C}^{\kappa,\upsilon}$). We use here a slightly different convention. We choose $\alpha > 0$ and $\beta \geq 1$ and take $M_k = \exp\p{\frac{\alpha k^\beta}{\beta}}$. We denote by $\mathcal{C}^{\alpha,\beta}$ the class $\mathcal{C}^M$. Notice that when $\beta = 1$, we find the class of real-analytic functions (for any value of $\alpha$). Notice also that when $\beta > 2$, the class $\mathcal{C}^{\alpha,\beta}$ is not closed under differentiation. We assume now that $T'$ belongs to the class $\mathcal{C}^{\alpha,\beta}$.

Let us deal first with the case $ 1 < \beta < 2$, then we see that for $0 < x < 1$ and some constants $c$ depending on $\alpha$ we have ($k$ is still from the proof of Lemma \ref{lmpropw})
\begin{equation}\label{eqk}
\begin{split}
k(x) \geq c^{-1} \va{\log x}^{\frac{1}{\beta - 1}} - c,
\end{split}
\end{equation}
thus if $\mu \in \left]0,1\right[$, we have for some new constant $c>0$ and small $x > 0$
\begin{equation}\label{eqwmu}
\begin{split}
\frac{w(\mu x)}{w(x)} \leq c\exp\p{-c^{-1}\ \va{\log x}^{\frac{1}{\beta - 1} }}.
\end{split}
\end{equation}
Then, for some new constant $c > 0$ the estimates on the singular values of $\mathcal{L}$ from Theorem \ref{thm1quant} becomes (for $k \geq 1$)
\begin{equation*}
\begin{split}
\sigma_k \leq c \exp\p{- c^{-1} \p{\log k}^{\frac{1}{\beta - 1}}}.
\end{split}
\end{equation*}
Once again, this gives, with Remark \ref{rmqgenuszero}, that, up to taking larger $c$,
\begin{equation}\label{eqagain}
\begin{split}
\va{d(z)} \leq \prod_{k \geq 1} \p{1 + c \va{z}\exp\p{- c^{-1} \p{\log k}^{\frac{1}{\beta - 1}}}}.
\end{split}
\end{equation}
Using \cite[Lemma 3.5.1]{Boas} to bound the right hand side of \eqref{eqagain} (that is using \eqref{eqwk} with the modification described in Remark \ref{rmqgenuszero}), we get that for some new constant $c > 0$ and all $z \in \C$ we have
\begin{equation}\label{eqtjs}
\begin{split}
\log_+ \va{d(z)} \leq c \exp\p{c \p{\log_+ \va{z}}^{\beta - 1}}.
\end{split}
\end{equation}
In particular, $d$ has order zero, but this could have been seen as a consequence of Corollary \ref{corordre}.

Now, if $\beta = 2$ then we have
\begin{equation*}
\begin{split}
\lim_{k \to + \infty} \p{\frac{(k+1) M_{k+1}}{M_k}}^{\frac{1}{k}} = e^{\alpha}.
\end{split}
\end{equation*}
Thus, by Corollary \ref{corordre}, the dynamical determinant $d$ has order less than $\frac{\alpha}{\log \lambda}$. We have here a very interesting behaviour: the bound on the order of the dynamical determinant depends on the expansion factor (this implies in particular that trace formula holds for large iterates of $\mathcal{L}$, see \cite[Theorem 2.4 and Remark 2.5]{lagtf}). As pointed out in Remark \ref{rmqordre}, it is not surprising that this behaviour occurs for the value of $\beta$ that separates classes that are stable under differentation and those that are not. As far as we know, it is the first time that such a behavior is proved and, consequently, it would be particularly interesting to know wether our result is sharp or not in that case.

Finally, we deal with the case $\beta > 2$. The estimates \eqref{eqk}, and thus \eqref{eqwmu}, remain true. Thus, for some $c > 0$, we have for large $N$ (recall that $g$ is defined by \eqref{eqdefg})
\begin{equation*}
\begin{split}
g(N) \leq c \exp\p{-c^{-1} N^{\frac{1}{\beta - 1}}}
\end{split}
\end{equation*}
and thus, changing the value of $c$,
\begin{equation*}
\begin{split}
\prod_{k=0}^{N-1} g(N) \leq c \exp\p{-c^{-1} N^{1 + \frac{1}{\beta - 1}}}.
\end{split}
\end{equation*}
Then, in the definition \eqref{eqdefFetG} of $G$, we may split the sum between $n \leq \p{\frac{2 \log r}{c}}^{\beta-1}$ and $n > \p{\frac{2 \log r}{c}}^{\beta-1}$, to find that for some $c > 0$ and all $r>0$
\begin{equation*}
\begin{split}
\log_+ G(r) \leq c \p{\p{\log_+ r}^{\beta} + 1}
\end{split}
\end{equation*}
An easy computation shows that for some $c > 0$ and all $m \geq 1$ we have
\begin{equation*}
\begin{split}
f(m) \leq c \exp\p{- c^{-1} \p{\log m}^{\frac{\beta}{\beta - 1}}},
\end{split}
\end{equation*}
where $f$ has been defined by \eqref{eqdeff}. Thus reasoning as above in the case $\beta < 2$ (that is using \cite[Lemma 3.5.1]{Boas}, which has been stated as \eqref{eqwk} in this case), we find that for some $c > 0$ and $r > 0$
\begin{equation*}
\begin{split}
\log_+ F(r) \leq c \exp\p{c \p{\log_+ r}^{\frac{\beta- 1}{\beta}}}.
\end{split}
\end{equation*}
And by Proposition \ref{propbound}, we see that there is still a new constant $c > 0$ such that for all $z \in \C$ we have
\begin{equation}\label{eqtjs2}
\begin{split}
\log_+ \va{d(z)} \leq c \exp\p{c \p{\log_+ \va{z}}^{\beta - 1}}.
\end{split}
\end{equation}
Notice that this is the same estimate than \eqref{eqtjs} that we established in the case $\beta < 2$, and that it is still true in the case $\beta = 2$ (but we have more precise information in this case). It is very interesting that the bound \eqref{eqtjs2} is true regardless of the value of $\beta$ while there is a huge change in the structure of the transfer operator at $\beta = 2$. Hence, it seems that in most cases the nuclear power decomposition contains all the information that we need on the dynamical determinant. This is indeed a very versatile tool that allows also to deal with finitely differentiable map \cite{Bal2}, and as we have just seen, it does not seem that we lose much information by using this method in more favorable cases. Notice however that in some \emph{very} favorable cases (such as Gevrey and analytic dynamics), the nuclear decomposition does not seem to give the best bound (this is because in this case, the bounds on the singular values of $\mathcal{L}$ and $\mathcal{L}_c$ are very similar).

\section*{Appendix: Weighted transfer operators}

It is sometimes useful to consider more general tranfer operators that the one defined by \eqref{eqdefL}. If $\psi : \mathbb{S}^1 \to \C$ is a weight we may define the weighted transfer operator $\mathcal{L}_\psi$ by
\begin{equation*}
\begin{split}
\mathcal{L}_\psi \varphi : x \mapsto \sum_{y : Ty = x} \frac{\psi(y)}{\va{T'(y)}} \varphi(y).
\end{split}
\end{equation*}

We shall assume in the following that $\psi$ is of class $\mathcal{C}^M$. It is then easy to see that the analysis above remains true for the operator $\mathcal{L}_\psi$, so that we can state:

\begin{prop}\label{propgen}
Theorem \ref{thm1quant} and Corollary \ref{corordre} remains true when $\mathcal{L}$ is replaced by $\mathcal{L}_\psi$. Moreover, we may also define in this case the decomposition \eqref{eqnpd}. This decomposition satisfies Propositions \ref{propfastdecay} and \ref{propnucpart}. Propositions \ref{propfact} and \ref{propbound} remains true as well if we replace the dynamical determinant $d$ by $d_\psi$ which is obtained from \eqref{eqdefdet} by replacing $\textup{tr}^\flat\p{\mathcal{L}^n}$ by
\begin{equation*}
\begin{split}
\textup{tr}^\flat\p{\mathcal{L}_\psi^n} = \sum_{x:T^n x =x} \frac{\prod_{k=0}^{n-1} \psi(T^k x)}{\va{1 - \p{T^n}'(x)}}.
\end{split}
\end{equation*}
\end{prop}

To prove Proposition \ref{propgen}, notice that the actual definition of $\mathcal{L}$ was only used in the proofs of Lemma \ref{lmfonda} and Proposition \ref{propfact} in the analysis above. The computation that gave Proposition \ref{propfact} can still be carried out and will give the formula that we announced for the flat trace of the weighted transfer operator. Thus we shall only explain how we can replace the operator $\mathcal{L}$ by $\mathcal{L}_\psi$ in the proof of Lemma \ref{lmfonda}.

\begin{lm}
Lemma \ref{lmfonda} remains true when $\mathcal{L}$ is replaced by $\mathcal{L}_\psi$.
\end{lm}

\begin{proof}
Recall the differential operator $L_{a_{k,\ell}}$ introduced in the proof of Lemma \ref{lmfonda} and notice that for all $m \in \N$ we have
\begin{equation*}
\begin{split}
\langle \mathcal{L}_\psi e_\ell, e_k \rangle_{L^2} = \int_{\mathbb{S}^1} e^{2 i \pi (\ell x - k T(x))} \psi(x) \mathrm{d}x = \int_{\mathbb{S}^1} e^{2 i \pi (\ell x - k T(x)} L_{a_{k,\ell}}^m(\psi)(x) \mathrm{d}x,
\end{split}
\end{equation*}
and thus we want to bound $\n{L_{a_{k,\ell}}^m(\psi)}_{\infty}$ instead of $\n{L_{a_{k,\ell}}^m(1)}_{\infty}$. As in the proof of Lemma \ref{lmfonda}, we notice that there are natural integer coefficients that do not depend on $a_{k,\ell}$ nor $\psi$ such that 
\begin{equation}\label{eqformule}
\begin{split}
L_{a_{k,\ell}}^m(\psi) = \sum_{n_1 + \dots + n_m + k = m} c_{n_1,\dots,n_m,k} \psi^{(k)} \prod_{j=1}^m a_{k,\ell}^{(n_j)}.
\end{split}
\end{equation}
Then, working as in the proof of Lemma \ref{lmfonda} and using the fact that $\psi$ is of class $\mathcal{C}^M$, we find constants $C,R > 0$ that do not depend on $m, k$ or $\ell$ such that
\begin{equation*}
\begin{split}
\n{L_{a_{k,\ell}}^m(\psi)}_{\infty} \leq C \p{\frac{R}{\va{k}}}^m M_m \sum_{n_1 + \dots + n_m + k = m} c_{n_1,\dots,n_m,k} k! \prod_{j=1}^m n_j!.
\end{split}
\end{equation*}
As in the proof of Lemma \ref{lmfonda}, we introduce now the operator $L_a$ obtained by replacing the function $a_{k,\ell}$ in the definition of $L_{a_{k,\ell}}$ by $a : x \mapsto \frac{1}{1-x}$. Since the coefficients in \eqref{eqformule} do not depend on $a_{k,\ell}$ nor $\psi$ we have
\begin{equation*}
\begin{split}
L_a^m(a)(0) = \sum_{n_1 + \dots + n_m + k = m} c_{n_1,\dots,n_m,k} k! \prod_{j=1}^m n_j!
\end{split}
\end{equation*}
but direct computation shows that $L_a^m(a) : x \mapsto \frac{2^m m!}{(1-x)^{2m+1}}$, and this ends the proof.
\end{proof}

\bibliographystyle{plain}
\bibliography{biblio_op_DC.bib}

\end{document}